\RequirePackage{fix-cm}
\documentclass[smallextended,final,11pt]{svjour3} 
\smartqed 
\usepackage[utf8]{inputenc}
\makeatletter
\def\cl@chapter{}
\makeatother
\DeclareMathAlphabet{\altpazocal}{OMS}{cmsy}{m}{n}
\usepackage{newtxmath,newtxtext}   

\journalname{Numerical Algorithms}
\spnewtheorem{fact}{Fact}{\bfseries}{\it}
\usepackage[misc]{ifsym}

\usepackage[shortlabels]{enumitem}
\usepackage{mathtools}
\usepackage{hyperref}
\usepackage[capitalize]{cleveref}
\usepackage{autonum}
\usepackage{cite}
\usepackage{color}
\usepackage{float}
\usepackage{subcaption}
\usepackage{graphicx}

\usepackage{url}
\usepackage[output-decimal-marker={.}]{siunitx}
\usepackage{tabularx,booktabs}
\usepackage[left,mathlines]{lineno}
\usepackage{etoolbox} 

\newcommand*\linenomathpatchAMS[1]{%
  \expandafter\pretocmd\csname #1\endcsname {\linenomathAMS}{}{}%
  \expandafter\pretocmd\csname #1*\endcsname{\linenomathAMS}{}{}%
  \expandafter\apptocmd\csname end#1\endcsname {\endlinenomath}{}{}%
  \expandafter\apptocmd\csname end#1*\endcsname{\endlinenomath}{}{}%
}

\expandafter\ifx\linenomath\linenomathWithnumbers
  \let\linenomathAMS\linenomathWithnumbers
  \patchcmd\linenomathAMS{\advance\postdisplaypenalty\linenopenalty}{}{}{}
\else
  \let\linenomathAMS\linenomathNonumbers
\fi

\linenomathpatchAMS{gather}
\linenomathpatchAMS{multline}
\linenomathpatchAMS{align}
\linenomathpatchAMS{alignat}
\linenomathpatchAMS{flalign}

\usepackage{fancyhdr}

\usepackage{dsfont}

\usepackage[margin=1in]{geometry}

\newcommand{\fenv}[1]%
{\ensuremath{\,\overrightarrow{\operatorname{env}}_{#1}}}
\newcommand{\benv}[1]%
{\ensuremath{\,\overleftarrow{\operatorname{env}}_{#1}}}

\newcommand{\RR}{\ensuremath{\mathds R}}

\newcommand{\NN}{\ensuremath{\mathds N}}

\newcommand{\aff}{\ensuremath{\operatorname{aff}}}

\newcommand{\Fix}{\ensuremath{\operatorname{Fix}}}

\newcommand{\Id}{\ensuremath{\operatorname{Id}}}

\DeclareMathOperator{\circum}{circumcenter}
\DeclareMathOperator{\dom}{dom}
\DeclareMathOperator{\dist}{dist}
\newcommand{\MAP}{{\rm MAP}}
\newcommand{\DRM}{{\rm DRM}}
\newcommand{\CRM}{{\rm CRM}}

\usepackage{calrsfs}
\DeclareMathAlphabet{\pazocal}{OMS}{zplm}{m}{n}

\newcommand{\SOC}{\mathcal{C}}

{\begin{list}{}{%
\settowidth{\labelwidth}{\textrm{#1~}}%
\setlength{\leftmargin}{\labelwidth+\labelsep}}}
{\end{list}}

\renewcommand\theenumi{(\roman{enumi})}
\renewcommand\theenumii{(\alph{enumii})}

\newcommand{\CC}{\ensuremath{C}}

\newcounter{count}

\allowdisplaybreaks
\begin{document}

\title{On the Circumcentered-Reflection Method for the Convex Feasibility Problem\thanks{RB was partially supported by the \emph{Brazilian Agency  Conselho Nacional de Desenvolvimento Cient\'ifico e Tecno\'ogico} (CNPq), Grants 304392/2018-9 and 429915/2018-7; \\ YBC was partially supported by the \emph{National Science Foundation} (NSF), Grant DMS -- 1816449.}}
\author{Roger Behling\and 
Yunier Bello-Cruz \and Luiz-Rafael Santos}


\institute{
  Roger Behling \Letter   \at School of Applied Mathematics, Funda\c{c}\~ao Get\'ulio Vargas \\ 
Rio de Janeiro, RJ -- 22250-900, Brazil. \email{rogerbehling@gmail.com}
    \and 
Yunier Bello-Cruz \at Department of Mathematical Sciences, Northern Illinois University. \\  DeKalb, IL -- 60115-2828, USA. \email{yunierbello@niu.edu}
\and 
 Luiz-Rafael Santos \at Department of Mathematics, Federal University of Santa Catarina. \\ 
Blumenau, SC -- 88040-900, Brazil. \email{l.r.santos@ufsc.br}
}

\date{\today}
\maketitle

\begin{abstract} 

The ancient concept of circumcenter has recently given birth to the Circumcentered-Reflection method (CRM). CRM was first employed to solve best approximation problems involving affine subspaces. In this setting, it was shown to outperform the most prestigious projection based schemes, namely, the Douglas-Rachford method (DRM) and the method of alternating projections (MAP). We now prove convergence of CRM for finding a point in the intersection of a finite number of closed convex sets. This striking result is derived under a suitable product space reformulation in which a point in the intersection of a closed convex set with an affine subspace is sought. It turns out that CRM skillfully tackles the reformulated problem. We also show that for a point in the affine set the CRM iteration is always closer to the solution set than both the MAP and DRM iterations. Our theoretical results and numerical experiments, showing outstanding performance, establish CRM as a powerful tool for solving general convex feasibility problems.



\end{abstract}

\section{Introduction}\label{sec:intro}

We consider the important problem of finding a point in a nonempty set $X\subset \RR^n$ given by the intersection of finitely many closed convex sets $\{X_i\}_{i=1}^m$, that is, 
\[\label{OriginalProblem}
\text{ Find } x^{*}\in X\coloneqq  \bigcap_{i=1}^m X_i. 
\]
We assume that for each $i=1,2,\ldots,m$ and any $z\in\RR^n$, the orthogonal projection of $z$ onto $X_i$, denoted by $P_{X_i}(z)$, is available.

It is well known that Problem \eqref{OriginalProblem} arises in many applications in science and engineering and is often solved enforcing the Douglas-Rachford method (DRM) or the method of alternating projections (MAP). The extensive literature on DRM and MAP in various contexts of Continuous Optimization expresses their relevance in the field (see, for instance, \cite{Bauschke:2006ej,BCNPW15,Phan:2016hl,Svaiter:2011fb,Bauschke:2016bt,Douglas:1956kk,Bauschke:2016jw,Bauschke:2016,BCNPW14}). Notwithstanding, DRM and MAP sequences may converge slowly due to spiraling and zigzag behavior, respectively, even in the case where the $X_i$'s are affine subspaces. This was our main motivation in \cite{Behling:2017da} for the introduction   of the Circumcentered-Reflection method (CRM).

The iterates of CRM are based on a generalization of the Euclidean concept of circumcenter (the point in the plane equidistant to the vertices of a given triangle). Successfully applied for projecting a given point onto the intersection of finitely many  affine subspaces (see \cite{Behling:2017da,Behling:2017bz,Behling:2019dj}), CRM is in its original form likely to face issues when dealing with general convex feasibility. Indeed, in the very first paper introducing CRM \cite{Behling:2017da}, it was pointed out that under non-affine structures, the method could possibly diverge or simply be undefined. There is now an actual example featuring two intersecting balls for which CRM stalls or diverges depending on the initial point (see \cite[Figure 10]{AragonArtacho:2019ug}). These apparent drawbacks are genuinely overcome in the present work. The key is to reformulate the problem of finding a point in $X$. The product space reformulation introduced by Pierra~\cite{Pierra:1984hl} will be considered. 

Product space versions of Problem \eqref{OriginalProblem} have been considered for both DRM and MAP~\cite{Bauschke:2006ej,AragonArtacho:2019ug}. Although necessary for DRM to converge if $m\geq 3$, the product space approach does not lead to satisfactory performance in comparison to variants of DRM  (see, for instance, the  Cyclic  Douglas--Rachford Method (CyDRM)~\cite{Borwein:2014,Borwein:2015vm} and the Cyclically Anchored  Douglas--Rachford Algorithm (CADRA)~\cite{Bauschke:2015df}). On the other hand, MAP converges with or without the product space reformulation, but tends to get worse on complexity when applied to the reformulated problem. Nonetheless, usually avoided for DRM and MAP, we will see that when suitably utilized, the product space reformulation captures CRM's essence and enables it to efficiently solve Problem \eqref{OriginalProblem}.

The main result in our work guaranties that, for any starting point, CRM converges to a solution of Problem \eqref{OriginalProblem}. This is one of the most impressive abilities of CRM proven so far, although noticeable results have already been derived by various authors. The history of circumcenters dates back to as early as $300$ BC, when they were described in Euclid's Elements \cite[Book 4, Proposition 5]{Euclid:1956vn}. More than two thousand years later, in 2018, circumcenters were discovered to be a simple yet effective way of accelerating the prominent Douglas-Rachford method~\cite{Behling:2017da}. In 2019, the paper celebrating 60 years of DRM \cite{Lindstrom:2018uc} mentions circumcenters as a natural way of dealing with DRM's spiraling characteristic. Also, CRM was employed for multi-affine set problems in \cite{Behling:2017bz} and in a block-wise version in \cite{Behling:2019dj}. Along with these articles, groups of researchers have been leading careful studies on properties, properness and calculations of circumcenters including a viewpoint of isometries (see \cite{Bauschke:2018ut,Bauschke:2019uh,Ouyang:2018gu,Bauschke:2018wa,Bauschke:2019}). Very recently, a work on CRM for particular nonconvex wavelet problems was carried out~\cite{Dizon:2019vq}. It seems that circumcenters are here to stay.

Before providing all the technical machinery of our approach, let us consider \Cref{fig:intro-comparison}  as a synthesized preview of what is developed in the present work. \Cref{fig:intro-comparison} illustrates the problem of finding a point in the intersection of two given balls and displays DRM, MAP and CRM trajectories starting all from the common initial point $x^0$ and the number of iterations taken by them to track a solution up to a given accuracy. A trajectory of CRM under the product space reformulation is exhibited as well and labeled as CRM-prod. 
After the picture, the definitions of  MAP, DRM, CRM and CRM-prod sequences are described. The particular intersection problem considered in  \Cref{fig:intro-comparison} is of special interest as for a different choice of starting point $x^{0}$,  CRM might not be well-defined or worse, it may diverge. In this paper, these two issues are overcome by CRM-prod.
\begin{figure}[H]\centering
\centering
\includegraphics[width=.60\textwidth]{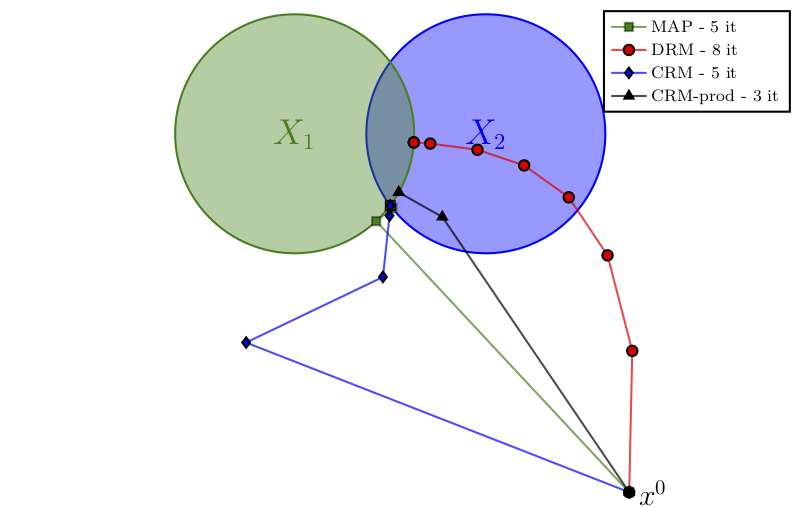}
   \caption{Geometric illustration of MAP, DRM, CRM and CRM-prod.}
  \label{fig:intro-comparison}
\end{figure}

Let us explain how the MAP, DRM and CRM sequences are generated for the two-set case, that is, the case in which a  point common to two closed convex sets $X_{1},X_{2}\subset\RR^n$ is sought. For a current iterate $z\in\RR^n$, we move to $z_{\MAP}$, $z_{\DRM}$, $z_{\CRM}$ using MAP, DRM and CRM, respectively, where
\[\label{MAP_DRM_CRM}
z_{\MAP}\coloneqq  P_{X_{2}}P_{X_{1}}(z), \qquad z_{\DRM}\coloneqq  \frac{1}{2}(\Id+R_{X_{2}}R_{X_{1}})(z), \qquad z_{\CRM}\coloneqq  \circum\{z,R_{X_{1}}(z),R_{X_{2}}R_{X_{1}}(z)\}.
\]
Let $\Id$ denote the identity operator in $\RR^n$ and consider the notation for reflection operators $R_{\Omega}\coloneqq  2P_{\Omega}-\Id$ for a given nonempty, closed and convex set $\Omega\subset\RR^n$. Note that MAP employs a composition of projections and, not by chance, DRM is also known as averaged reflection method, as it iterates by taking the mean of the current iterate with a composition of reflections. As for CRM, it chooses the Euclidean circumcenter of the triangle immersed in $\RR^n$ with vertices $z$, $R_{X_{1}}(z)$ and $R_{X_{2}}R_{X_{1}}(z)$. The circumcenter $z_{\CRM}$ is defined by satisfying two criteria: (i) $z_{\CRM}$ is equidistant to the three vertices, i.e., $\|z-z_{\CRM}\|=\|R_{X_{1}}(z)-z_{\CRM}\|=\|R_{X_{2}}R_{X_{1}}(z)-z_{\CRM}\|$, with $\|\cdot\|$ representing the Euclidean norm; (ii)  $z_{\CRM}$ belongs to the affine subspace determined by $z$, $R_{X_{1}}(z)$ and $R_{X_{2}}R_{X_{1}}(z)$, denoted here by $\aff\{z,R_{X_{1}}(z),R_{X_{2}}R_{X_{1}}(z)\}$.

In order to understand how CRM-prod works, we have to consider the product space environment by Pierra~\cite{Pierra:1984hl}. Assume that our problem still consists of finding a point in $X_{1}\cap X_{2}$ and take into account now two new closed convex sets $W,D\subset\RR^{2n}$, where $W\coloneqq  X_{1}\times X_{2}$ and $D\coloneqq \{(x,x)\in\RR^{2n} \mid x\in\RR^{n}\}$. Indeed, $W$ and $D$ are closed, convex and actually it is straightforward to see that $D$ is a subspace of $\RR^{2n}$. Then, it obviously holds that
\[\label{ProductSpaceEnvironment}
x^{*}\in X_{1}\cap X_{2} \Leftrightarrow (x^{*},x^{*})\in W\cap D.
\]
CRM-prod is ruled by $z^{k+1}\coloneqq  \circum\{z^k,R_W(z^k),R_DR_W(z^k)\}$. We anticipate that in our approach it will be key to initialize CRM-prod in $D$. Indeed, it will be shown that if $z^0\in D$, then $z^k\in D$ for all $k$. This enables us to fully understand the CRM-prod trajectory plotted in Figure 1. We took the initial point $z^0=(x^0,x^0)\in\RR^4$ so that the CRM-prod sequence lies entirely in $D$, that is, each $z^k\in\RR^4$ has the form $z^k=(x^k,x^k)$. So, we plotted the $x^k$ part of the CRM-prod sequence in \Cref{fig:intro-comparison}. 

With all these basic notions having been introduced, we get a clear geometric comprehension of the sequences illustrated in \Cref{fig:intro-comparison}. Now, it is important to stress that the intersection problem \eqref{OriginalProblem} when $m\geq 3$ can also be reformulated as the problem of finding a point in the intersection of two closed convex sets, one of which being a subspace (see  \Cref{sec:prod}). With that said, our investigation will focus on CRM for finding a point in the intersection of a closed convex set $K$ and an affine subspace $U$, a problem that is actually interesting and relevant on its own.

The ability of CRM for finding a point in the intersection of a closed convex set $K$ and an affine subspace $U$ is simply impressive in comparison to the classical MAP and DRM. Moreover, a CRM iteration is always better by means of distance to the solution set than both DRM and MAP iterations taken at any point in $U$ (see the last theorem in Section \ref{sec:convex-affine}).  Also, it has to be noted that the computational effort for calculating a CRM step is essentially the same as the one for computing a MAP or DRM step. This is due to the fact that a circumcenter outcoming from a two-set intersection problem consists of solving a $2\times 2$ linear system of equations. The outstanding numerical performance of CRM over MAP and DRM recorded in our experiments not only reveals a Newtonian flavor of CRM, it also opens opportunities for new research.  Moreover, an explicit connection with Newton-Raphson method was shown in \cite{Dizon:2019vq} for nonconvex problems, and used to prove quadratic convergence of a hybrid circumcenter scheme.

Our paper is organized as follows. Section \ref{sec:convex-affine} is at the core of this work containing our two most technical results, namely \Cref{TeoremaSpan} and \Cref{ComparisonCRM_MAP_DRM}. \Cref{TeoremaSpan} states convergence of CRM when applied to finding a point common to an affine subspace and a general closed convex set. \Cref{ComparisonCRM_MAP_DRM} claims that CRM is better than both MAP and DRM when computed at iterates belonging to the affine subspace. Also, \Cref{TeoremaSpan}, together with the product space reformulation, leads to the main contribution of the manuscript in Section \ref{sec:prod}, namely, the global convergence of CRM for solving Problem \eqref{OriginalProblem}. Numerical experiments are conducted in Section \ref{sec:NI} concerning convex inclusions such as polyhedral and second-order cone feasibility. The results show that in all instances, CRM outperforms MAP and DRM. We close the article in Section \ref{sec:concluding} with a summary of our results together with perspectives for future work.

\section{CRM for the intersection of a convex set and an affine subspace}\label{sec:convex-affine}

In this section, the problem under consideration is given by
\[\label{eq:KcapUProblem}
\text{ Find }z^{*}\in K\cap U,
\]
with $K\subset \RR^{n}$ a closed convex set, $U\subset \RR^{n}$ an affine subspace and $K\cap U$ nonempty.

We are going to prove that CRM solves problem \eqref{eq:KcapUProblem} as long as the initial point lies in $U$ and the circumcenter iteration considers first a reflection onto $K$ and then onto $U$. This CRM scheme reads as 
\[\label{eq:CRMKcapUiter}
z^{k+1} \coloneqq  \CC(z^{k})= \circum\{z^{k},R_{K}(z^{k}),R_{U}R_{K}(z^{k})\}, \text{ with } z^{k}\in U.
\]
It turns out that if $z^0\in U$, then the whole sequence $(z^k)_{k\in\NN}$ generated upon \eqref{eq:CRMKcapUiter} remains in $U$ and converges to a solution, that is, a point in $K\cap U$. This result lies at the core of our work. In order to establish it, we need to go through some preliminaries. We start with a formal and general definition of circumcenter.

\begin{definition}[circumcenter operator]\label{CircDef}
Let $\pazocal{B}\coloneqq (Y_1,Y_2,\ldots,Y_q)$ be a collection of ordered nonempty closed convex sets in $\RR^{n}$, where $q\ge 1$ is a fixed integer. 
The circumcenter  of $\pazocal{B}$ at the point $z\in \RR^n$  is denoted by $C_\pazocal{B}(z)$ and defined by the following properties:
	\item [ \bf (i)] $\|z-C_\pazocal{B}(z)\|=\|R_{Y_1}(z)-C_\pazocal{B}(z)\|=\cdots=\|R_{Y_q}\cdots R_{Y_2}R_{Y_1}(z)-C_\pazocal{B}(z)\|$; and 
	\item [ \bf (ii)] $C_\pazocal{B}(z)\in \aff\{z,R_{Y_1}(z),R_{Y_2}R_{Y_1}(z),\ldots ,R_{Y_q}\cdots R_{Y_2}R_{Y_1}(z)\}$.
\end{definition}
For convenience, we sometimes also write $C_{\pazocal{B}}(z)$ as \[\circum\{z,R_{Y_1}(z),R_{Y_2}R_{Y_1}(z),\ldots ,R_{Y_q}\cdots R_{Y_2}R_{Y_1}(z)\}.\]
Circumcenters, when well-defined, arise from the intersection of suitable bisectors and their computation requires the resolution of a $q\times q$ linear system of equations~\cite{Bauschke:2018ut}, with $q$ as in Definition \ref{CircDef}. 
The good definition of circumcenters is not an issue when the convex sets in question are intersecting affine subspaces (see \cite{Behling:2017bz}). This  might not be the case for general convex sets (see \cite{Behling:2017da}).
However, we are going to see  that for any point $z\in U$ the circumcenter used in \eqref{eq:CRMKcapUiter}, namely  $\CC(z) =\CC_{\pazocal{B}}(z)$, where $\pazocal{B} = (K, U)$,  is well-defined.

We proceed now by listing and establishing results that are going to enable us  to indeed prove that $C(z)$ is well-defined for all $z\in U$.

\begin{lemma}[good definition and characterization of CRM for intersecting affine subspaces] \label{charact_circ}
Consider a collection of affine subspaces $\pazocal{B}=(U_1,U_2,\ldots,U_q)$ with nonempty intersection $U_\pazocal{B}\coloneqq  \bigcap_{i=1}^q U_i$. For any $z\in\RR^n$, $C_\pazocal{B}(z)$ exists, is unique and fulfills
	\item [ \bf (i)] $P_{U_\pazocal{B}}(C_\pazocal{B}(z))=P_{U_\pazocal{B}}(z)$; and 
	\item[ \bf (ii)]  for any $s\in U_\pazocal{B}$, we have $C_\pazocal{B}(z)=P_{U_z}(s)$, where $U_z\coloneqq  \aff\{z,R_{U_1}(z),R_{U_2}R_{U_1}(z),\ldots ,R_{U_q}\cdots R_{U_2}R_{U_1}(z)\}$.
\end{lemma}
\begin{proof}
See \cite[Lemmas 3.1 and 3.2]{Behling:2017bz}.
\qed\end{proof}

It is  worth emphasizing that, whenever we deal with circumcenter operators regarding two sets, many results can be   carried out as if we were in  $\RR^{2}$ since, in this case, the circumcenter is a point sought in the affine subspace defined by three points. So, this affine subspace has dimension of at most $2$. Along the manuscript, one is going to come across arguments  based on this fact.


\begin{lemma}[one step convergence of CRM for a hyperplane and an affine subspace]\label{lemma:CRMhyperplane}
Let $H,U\subset \RR^{n}$ be a hyperplane and an affine subspace, respectively. If $H\cap U$ is nonempty, then  
\[P_{H\cap U}(z) = \circum\{z,R_{H}(z),R_{U}R_{H}(z)\},\] for any $z\in U$.
\end{lemma}
\begin{proof} Let $z\in U$. If $z$ lies also in $H$, the result follows trivially. Therefore, assume $z\notin H$. For our analysis it will be convenient to look at the point $P_UR_{H}(z)$. If $P_UR_{H}(z)$ coincides with $z$, then $\aff\{z,R_{H}(z),R_{U}R_{H}(z)\}=\aff\{R_{H}(z),R_{U}R_{H}(z)\}$, because $z=P_UR_{H}(z)=\frac{1}{2}(R_{H}(z)+R_{U}R_{H}(z))$. Also, the only point equidistant to $R_{H}(z)$ and $R_{U}R_{H}(z)$ in $\aff\{R_{H}(z),R_{U}R_{H}(z)\}$ is $P_UR_{H}(z)=z$. This means that $z=\circum\{z,R_{H}(z),R_{U}R_{H}(z)\}$ and, in particular, $z=R_{H}(z)$, which implies that $z\in H$, a contradiction. So, we have that $z$ and $P_UR_{H}(z)$ are distinct and the line connecting these points, denoted here by $L_z$, is well-defined. Moreover, $L_z$ lies entirely in $U$ and perpendicularly crosses the segment $\overline{R_{H}(z)\,R_{U}R_{H}(z)}\coloneqq  \{w\in\RR^n: w=tR_{H}(z)+(1-t)R_{U}R_{H}(z),\, t\in [0,1]\}$ at its midpoint, namely $P_UR_{H}(z)$. So, $L_{z}$ is a bisector of the segment $\overline{R_{H}(z)\,R_{U}R_{H}(z)}$. Thus, $\tilde z\coloneqq \circum\{z,R_{H}(z),R_{U}R_{H}(z)\}$ has to lie in $L_z$ and consequently in $U$. 

Furthermore,  ${\displaystyle \overline{z\,P_H(z)}\perp \overline{\tilde z\,P_H(z)}}$ because the circumcenter is the point where the bisectors of the triangle of vertices $z,\,R_{H}(z)$ and $R_{U}R_{H}(z)$ intersect. Since $H$ is a hyperplane, the orthogonality between $\displaystyle \overline{z\,P_H(z)}$ and $\overline{\tilde z\,P_H(z)}$ suffices to guaranty that $\tilde z\in H$. Thus, $\tilde z\in H\cap U$. Finally, from Lemma \ref{charact_circ}, we have that $P_{H\cap U}(z)=P_{H\cap U}(\tilde z)$. Hence,  $P_{H\cap U}(z)=\tilde z$, which completes the  proof.  
\qed
\end{proof}

The next results concern the domain of the circumcenter operator for $K$  and $U$, namely $\dom(\CC)\coloneqq  \{z\in\RR^{n} \mid \CC(z)=\circum\{z,R_{K}(z),R_{U}R_{K}(z)\}\text{ is well-defined} \}$. We will see that $U\subset\dom(\CC)$ and $C(z)\in U$, whenever $z\in U$.

\begin{lemma}[well-definedness and characterization of CRM for $K$ and $U$]
\label{lemma:goodCRMKcapU} Let  $K,U\subset \RR^{n}$, where  $K$ is  a closed convex set and $U$ is  an affine subspace and   assume  $K\cap U$ to be nonempty. Then, for all $z\in U$, $\CC(z)\coloneqq  \circum\{z,R_{K}(z),R_{U}R_{K}(z)\}$ is well-defined and we have $C(z) \in U$. Furthermore, $\CC(z) = P_{H_z\cap U}(z)$, where $H_z\coloneqq    \{x\in \RR^{n}\mid (x-P_K(z))^{T}(z-P_K(z))= 0\}$ if $z\notin K$ and $H_z\coloneqq  K$, otherwise. 
\end{lemma}

\begin{proof}
Let $z\in U$. If $z\in K$, the result is trivial. So assume that $z$ belongs to $U$, but not to $K$. Then, we have of course that $P_{H_z}(z)=P_K(z)$ and thus $R_{H_z}(z)=R_K(z)$. Note that $H_{z}\cap U\neq \varnothing$. In fact, let $y\in K\cap U$. Thus, by $y \in K$ and by the characterization of projection on nonempty closed convex set, $  \left(z-P_{K}(z)\right)^{T} \left (y-P_{K}(z)\right) \leq 0$.  Define $f: [0,1] \rightarrow \RR$ by $  f(t)\coloneqq\left( z-P_{K}(z)\right)^{T} \left(t z+(1-t) y-P_{K}(z)\right)$. Then $f(0)= \left(z- P_{K}(z)\right)^{T} \left(y-P_{K}(z)\right) \leq 0$ and  $f(1)=\left( z-P_{K}(z)\right)^{T}\left(z-P_{K}(z)\right)=\left\|z-P_{K}(z)\right\|^{2}>0$. Since  $f$ is continuous,   there exists $\bar{t} \in [0,1[$   such that $f(\bar t) =  \left( z-P_{K}(z)\right)^{T} \left(\bar t z+(1-\bar t) y-P_{K}(z)\right) = 0$. Thus, 
$\bar{t} z+(1-\bar{t}) y \in H_{z}$.  On the other hand, as both $z $ and $ y $ are in $U$, an affine subspace, we have $\bar{t} z+(1-\bar{t}) y \in U$ as well. Altogether, $\bar{t} z+(1-\bar{t}) y \in H_{z} \cap U$.

Therefore, $\CC(z)=\circum\{z,R_{K}(z),R_{U}R_{K}(z)\}=\circum\{z,R_{H_z}(z),R_{U}R_{H_z}(z)\}=P_{H_z\cap U}(z)$, where the last equality follows by employing~\Cref{lemma:CRMhyperplane}.

\qed 
\end{proof}

Below, in \Cref{fig:illustration-KcapU}, we  illustrate geometrically  what has been established in \Cref{lemma:goodCRMKcapU}.
\begin{figure}[H]\centering
\centering
\includegraphics[width=.55\textwidth]{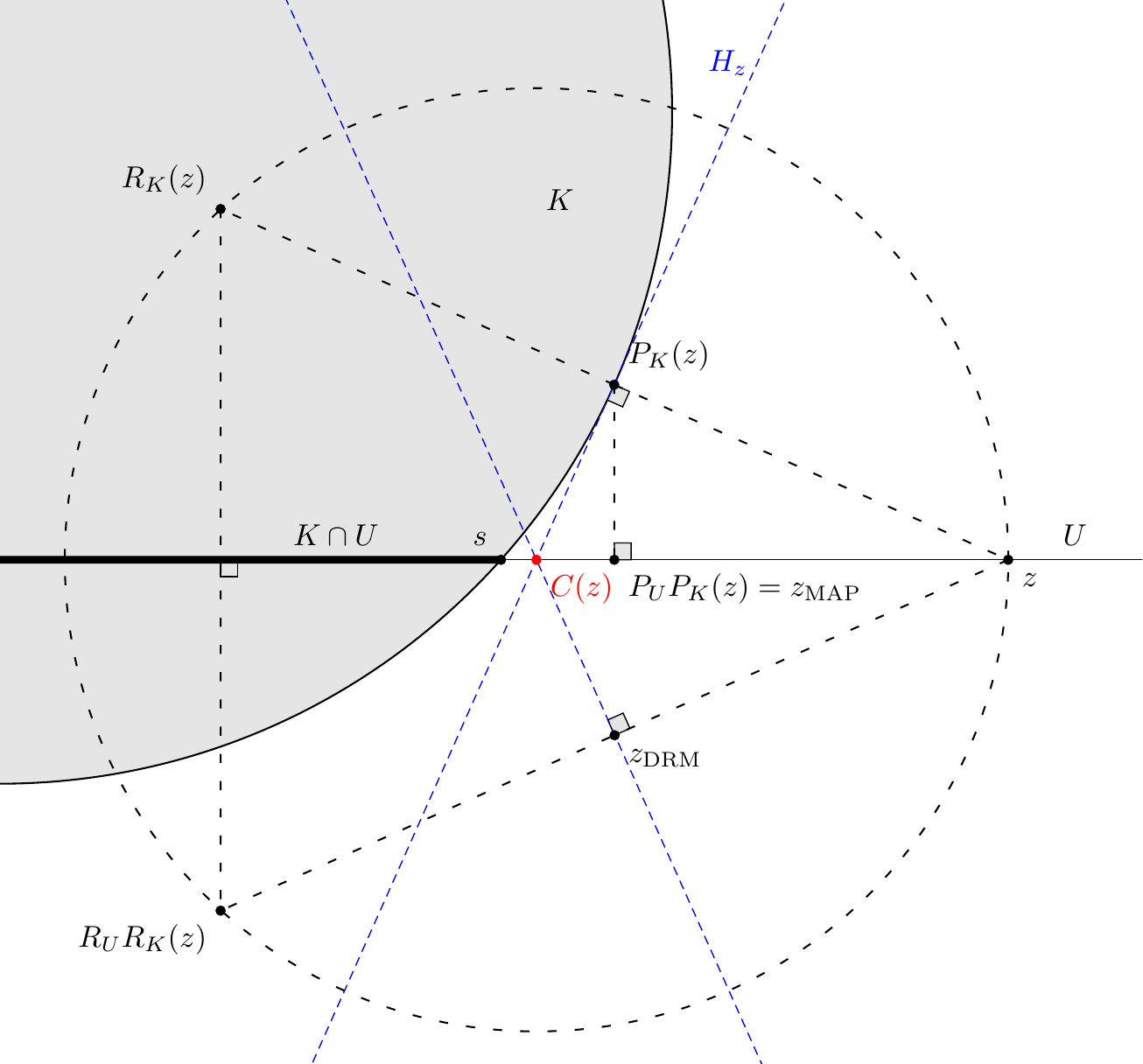}
   \caption{Illustration of CRM for the intersection between an affine $U$ and a convex $K$.}
  \label{fig:illustration-KcapU}
\end{figure}

Let us characterize the fixed point set of the circumcenter operator.

\begin{lemma}[fixed points of $\CC$]\label{Fix-CT} Assume $K,U\subset \RR^{n}$ as in~\Cref{lemma:goodCRMKcapU} and consider $\Fix \CC\coloneqq  \{z\in \dom(\CC)\mid C(z)=z\}$. Then,
\[\Fix \CC=K\cap U. \]
\end{lemma}
\begin{proof} Clearly if $z\in K\cap U$ then $z\in \Fix \CC$. Conversely, take $ z\in \Fix \CC$, that is, $\CC(z)=z$, or equivalently $z=R_K(z)=R_UR_K(z)$. So, $z\in K$ and then $z=R_U(z)$. Hence, $z\in U$, proving the lemma.

\qed
\end{proof}

Next, we derive a firmly nonexpansiveness property of the circumcenter operator $\CC$ restricted to $U$, known  as  firmly quasinonexpansiveness~\cite[Definition 4.1(iv)]{BC2011}.


\begin{lemma}[firmly quasinonexpansiveness of CRM]\label{lemma-quasinonexp}
Assume $K,U\subset \RR^{n}$ as in~\Cref{lemma:goodCRMKcapU}. Then, for any  $z\in U$ and  $s\in K\cap U$ 
\[\label{eq:lemmaqnonexpansCRM}\|\CC(z)-s\|^2 \le \|z-s\|^2-\|z-\CC(z)\|^2.
\] Moreover,
\[\label{eq:lemmaqnonexpansCRM2}\dist^2(\CC(z), K\cap U)\le \dist^2(z, K\cap U)-\|z-\CC(z)\|^2.
\]
\end{lemma}
\begin{proof}  Let $z\in U$. Define $H_z\coloneqq  \{h\in \RR^{n} \mid (h-P_K(z))^{T}(z-P_K(z))= 0\}$ and $H^+_z\coloneqq  \{h\in \RR^{n}\mid  (h-P_K(z))^{T}(z-P_K(z))\le 0\}$. It is clear that $K\subseteq H^+_z$, and so $K\cap U\subseteq H_z^+\cap U$. If $z\in H_{z}$, then $z\in K\cap U$ and, in this case, $\CC(z)= z$, in view of \Cref{Fix-CT}, and the claims follow easily.

 Let $z\notin H_z$.  We have that  $C(z)$ coincides with $\circum\{z,R_{H_z}(z),R_UR_{H_z}(z)\}$, since $R_{K}(z) = R_{H_{z}}(z)$. From \Cref{lemma:CRMhyperplane}, it follows that $C(z)=P_{H_z\cap U}(z)\in U$. As consequence of $z$ not being in $H_{z}$ we immediately get that $z$ cannot be in $K$. Also, it easily follows that $z\notin H_{z}^{+}$. This, combined with the fact that $H_{z}$ is the boundary of $H_{z}^{+}$, gives us  that $P_{H_z\cap U}(z)=P_{H^+_z\cap U}(z)$.  Now, taking  into account  that, for any $s\in K\cap U$, $P_{H^+_z\cap U}(s)=s$, we have 
\begin{align}\|C(z)-s\|^2 &=\|P_{H_z\cap U}(z)-s\|^2\\
&=\|P_{H^+_z\cap U}(z)-P_{H^+_z\cap U}(s)\|^2\\&\le \|z-s\|^2-\left \|[z-P_{H^+_z\cap U}(z)]-[s-P_{H^+_z\cap U}(s)]\right\|^2 \\
&= \|z-s\|^2-\|z-C(z)\|^2,
\end{align}
where the inequality is due to the firmly nonexpansiveness propriety of the projection~\cite[Proposition 4.16]{BC2011}. 
To prove the last part of the lemma, fix $s=\bar z\coloneqq P_{K\cap U}(z)$ and use \eqref{eq:lemmaqnonexpansCRM} in order to get 
\begin{align}\label{(4''')}\dist^2(C(z), K\cap U)&\le \|C(z)-\bar z\|^2 
\le \|z-\bar z\|^2-\|z-C(z)\|^2 \\ &
=\dist^2(z,K\cap U)-\|z-C(z)\|^2.
\end{align}
\qed
\end{proof}

We arrive now at the key result of our study, which establishes CRM as a tool for finding a point in $K\cap U$ whenever the initial point is chosen in $U$.

\begin{theorem}[convergence of CRM]\label{TeoremaSpan}
Assume $K,U\subset \RR^{n}$ as in \Cref{lemma:goodCRMKcapU} and let $z\in U$ be given. Then, the CRM sequence $(C^k(z))_{k\in\NN}$ is well-defined, contained in $U$ and converges to a point in $K\cap U$.
\end{theorem} 
\begin{proof} The well-definedness of $(C^k(z))_{k\in\NN}$ as well as its pertinence to $U$ are due to \Cref{lemma:goodCRMKcapU}. From \Cref{lemma-quasinonexp}, we have, for any $z\in U$, $s\in K\cap U$ and $\ell\in \NN$
\begin{align}
\|z^{\ell+1}-s\|^2&=\|C^{\ell+1}(z)-s\|^2 \\&= \|C(C^{\ell}(z))-s\|^2\\&\le\|C^{\ell}(z)-s\|^2 -\|C^{\ell}(z)-C(C^{\ell}(z))\|^2\\&=\|C^{\ell}(z)-s\|^2 -\|C^{\ell}(z)-C^{\ell+1}(z)\|^2\\&=\|z^{\ell}-s\|^2-\|z^\ell-z^{\ell+1}\|^2. \end{align} Hence, \[\label{MonotoneCRMseq}\|z^\ell-z^{\ell+1}\|^2\le \|z^{\ell}-s\|^2-\|z^{\ell+1}-s\|^2, \quad \forall \ell\in\NN.\] Summing from $\ell=0$ to $m$, we have
\[\sum_{\ell=0}^m\|z^\ell-z^{\ell+1}\|^2\le \sum_{\ell=0}^m (\|z^{\ell}-s\|^2-\|z^{\ell+1}-s\|^2)=\|z^{0}-s\|^2-\|z^{m+1}-s\|^2\le \|z^{0}-s\|^2.\] Taking limits as $m\to+\infty$, we get the summability of the associated series and so, 
$z^k-z^{k+1}$ converges to $0$ as $k\to+\infty$. 

Moreover, the sequence $(z^k)_{k\in\NN}$ is bounded because of \eqref{eq:lemmaqnonexpansCRM} in \Cref{lemma-quasinonexp}. That is, \[\|z^k-s\|=\|C^k(z)-s\|\le \|C^{k-1}(z)-s\|\le \cdots \le \|z-s\|.\] Let $\hat z$ be any cluster point of the sequence $(z^k)_{k\in\NN}$ and denote $(z^{i_k})_{k\in\NN}$ an associated convergent subsequence to $\hat z$. Note further that the fact $z^{i_k}-z^{i_k+1}\to 0$ implies $z^{i_k+1}\to \hat z$. We claim that $\hat z\in K\cap U$. Since $U$ is closed and $(z^k)_{k\in\NN}$ is contained in $U$, we must have $\hat z\in U$. By the definition of $C$ we have that
\[
\|z^{i_k}-z^{i_k+1}\|=\|z^{i_k}-C(z^{i_k})\|=\|C(z^{i_k})-R_K(z^{i_k})\|=\|z^{i_k+1}-R_K(z^{i_k})\|.
\] 
So, $z^{i_k+1}-R_K(z^{i_k})$ converges to $0$. It follows from the continuity of the reflection onto $K$ and taking limits as $k\to+\infty$ in the last equality that $\hat z=R_K(\hat z)$. Hence, $\hat z\in K$ proving the claim.


Therefore, so far we have that $(z^k)_{k\in\NN}$ is bounded, contained in $U$ and all its cluster points are in $K\cap U$. 
We could now complete the proof using standard Fej\'er monotonicity results (see \cite[Theorem 5.5]{BC2011} or \cite[Theorem 2.16(v)]{Bauschke:2006ej}). However, for the sake of self-containment, we show that all these cluster points are equal and hence $(z^k)_{k\in\NN}$ converges to a point in $K\cap U$. Let $\tilde{z}$, $\hat{z}$ be two 
accumulation points of $(z^k)_{k\in\NN}$ and $(z^{j_k})_{k\in\NN}$,
$(z^{i_k})_{k\in\NN}$ be subsequences  
convergent to $\tilde{z}$, $\hat{z}$ respectively. The real sequence $(\|z^k-\hat z\|)_{k\in\NN}$ is convergent because it is bounded below by zero and from \eqref{MonotoneCRMseq} it is monotone non-increasing. Thus, 
\[\|\tilde{z}-\hat{z}\|=\lim_{k\to+\infty}\|z^{j_k}-\hat z\|=\lim_{k\to+\infty}\|z^k-\hat z\|=\lim_{k\to+\infty}\|z^{i_k}-\hat z\|=0,\] establishing the desired result.


\qed\end{proof}

We present now a newsworthy nonconvex instance covered by \Cref{TeoremaSpan}.

\begin{remark}[line-sphere intersection]\label{remark:sphere}
A straightforward  consequence  of \Cref{TeoremaSpan} is the convergence of CRM for finding a point in the  intersection (when nonempty) of  a given closed ball and a line. 
Interestingly, this remark  almost promptly gives a convergence result for  CRM applied to solving the nonconvex intersection problem involving a line crossing a sphere. More precisely, for the latter, having an iterate on the line, CRM moves us out of the sphere and thereon CRM acts as if the sphere were a ball and convergence is guaranteed. An exception for moving out of the sphere occurs when the iterate is on the line and also happens to be the center of the sphere. In this case, due to nonconvexity, the projection operator is set-valued but one only needs to pick one of the two points in the direction of the line crossing the sphere as a projection and then CRM converges in one single step.  The relevance of this remark relies on the fact that the definitive proof of DRM for this problem took three papers to be completely derived~\cite{Benoist:2015,AragonArtacho:2013,Borwein:2011}.
\end{remark}

A nice consequence of \Cref{TeoremaSpan} regards convex intersection problems in which one of the sets is affine. 

\begin{corollary}[serial composition of circumcenters]\label{cor:serialconv}
 Let $K \coloneqq \bigcap_{i=1}^{N}K_{i}$, where $K_i$ is convex, for $i=1,\ldots,N$, and $U$ be an affine subspace and suppose $K\cap U\neq \varnothing$. Then, for all $z\in U$, 
 $\tilde C(z) \coloneqq C_{(K_{N}, U)} \circ \cdots\circ  C_{(K_{1}, U)}(z)$ is well-defined, with $C_{(K_{i}, U)} (z)$ being $\circum\{z,R_{K_{i}}(z),R_{U}R_{K_{i}}(z)\}$, for $i=1,\ldots, N$. Moreover, the CRM sequence $(\tilde C^{k}(z))_{k\in\NN}$ is well-defined, contained in $U$ and converges to a point in $K\cap U$.
\end{corollary}
\begin{proof}
We present the proof for when $N=2$ and an induction argument suffices for the general case. Thus, consider  $\tilde C(z) \coloneqq C_{(K_{2}, U)} \circ   C_{(K_{1}, U)}(z)$ and set $C_{1}\coloneqq C_{(K_{1}, U)} $, $C_{2}\coloneqq C_{(K_{2}, U)}$.  
Note that we can essentially proceed as in the proof of \Cref{Fix-CT} to show that
$\Fix \tilde C = K_{1}\cap K_{2} \cap U$.  In the following, we derive a similar inequality to \eqref{eq:lemmaqnonexpansCRM}, given in \Cref{lemma-quasinonexp}. Let $s\in K_{1}\cap K_{2} \cap U$ and $z\in U$.
Then, 
\[\begin{aligned}\|\tilde C(z)-s\|^{2} & = \| C_{2}(C_{1}(z))-s\|^{2} \\
& \leq\left\|C_{1}( z)-s\right\|^{2}-\left\|C_{1}( z)-\tilde C(z)\right\|^{2} \\ 
& \leq\|z-s\|^{2}-\left\|z-C_{1}( z)\right\|^{2}-\left\|C_{1}( z)-\tilde C(z)\right\|^{2} \\ 
&=\|z-s\|^{2}-  2\left[\frac{1}{2}\left\|z-C_{1}( z)\right\|^{2}+\frac{1}{2}\left\|C_{1}( z)-\tilde C(z)\right\|^{2}\right] \\ 
&=\|z-s\|^{2}-2\left[\left\|\frac{1}{2}\left(z-C_{1}( z)\right)+\frac{1}{2}\left(C_{1}( z)-\tilde C(z)\right)\right\|^{2}+\frac{1}{4}\|z - 2C_{1}(z) +  \tilde C(z)\|^{2}\right]\\
&\leq\|z-s\|^{2}-2\left\|\frac{1}{2}(z-\tilde C(z))\right\|^{2}\\
&=\| z-s\|^{2}-\frac{1}{2} \| z-\tilde C(z) \|^{2},
\end{aligned}
\]
where we used \Cref{lemma-quasinonexp} for the first two inequalities  and Corollary 2.15 by \cite{BC2011} for the third equality.
Now, by employing the same arguments of \Cref{TeoremaSpan}, we can prove that the sequence $(\tilde C^{k}(z))_{k\in\NN}$ is bounded and has an accumulation point  in $\Fix \tilde C$  and the  converge is achieved upon Fej\'er monotonicity. 
\qed
\end{proof}

\begin{remark} We can directly employ a similar argument of Corollary 4.48 in \cite{BC2011} to get the same claim as  in  \Cref{cor:serialconv}   replacing its serial composition  by a convex combination of circumcenter operators of the form  
$\hat C(z) \coloneqq \sum_{i=1}^{N}\omega_{i}C_{(K_{i}, U)}(z)$, with $\omega_i\in ]0,1[$, for $i=1,\ldots,N$ and  $\sum_{i=1}^{N}\omega_{i} = 1$.  
\end{remark}

We close this section with a theorem stating that for a given iterate in an affine subspace $U$, CRM gets us closer to the solution set than both MAP and DRM.

\begin{theorem}[comparing CRM with MAP and DRM]\label{ComparisonCRM_MAP_DRM}
	Assume $K,U\subset \RR^{n}$ as in \Cref{lemma:goodCRMKcapU} and let $z\in U$ be given. Also recall the notation $z_{\MAP}\coloneqq P_UP_K(z)$, $z_{\DRM}\coloneqq \frac{1}{2}(\Id+R_UR_K)(z)$ and $C(z)\coloneqq\circum\{z,R_{K}(z),R_{U}R_{K}(z)\}$. Then, for any $s\in K\cap U$ we have
	\item [ \bf (i)] $
	\|C(z)-s\|\leq\|z_{\MAP}-s\|\leq\|z_{\DRM}-s\|,
	$
	\item [ \bf (ii)]	$
	\dist(C(z), K\cap U)\leq \dist(z_{\MAP}, K\cap U)\leq \dist(z_{\DRM}, K\cap U).
	$	
\end{theorem} 
\begin{proof}
Assume $z\in U$ and $s\in K\cap U$ arbitrary but fixed. If $z\in K$ the result follows trivially because then $P_UP_K(z)=z$, $\frac{1}{2}(\Id+R_UR_K)(z)=z$ and $C(z)=z$ due to \Cref{Fix-CT}. So, let $z\in U\backslash K$ and also, in order to avoid translation formulas, let us assume without loss of generality that $U$ is a subspace. Recall that Lemma \ref{lemma:goodCRMKcapU} characterized $C(z)$ as the projection of $z$ onto the intersection of $U$ and the hyperplane with normal $z-P_K(z)$ passing through $P_K(z)$. Therefore, the triangle of vertices $z$, $P_K(z)$ and $C(z)$ has a right angle at $P_K(z)$ and the triangle of vertices $z$, $P_UP_K(z)$ and $P_K(z)$ has a right angle at $P_UP_K(z)$ since $U$ is a subspace. Considering these two triangles and taking into account the fact that hypotenuses are larger than  corresponding legs, we conclude that
\[\label{CRMstep_longer_thanMAP}
\|C(z)-z\|\geq\|P_K(z)-z\|\geq\|P_UP_K(z)-z\|.
\]

Another property we have is that the MAP point $P_UP_K(z)$ is a convex combination of $z$ and $C(z)$. In order to deduce this, first note that these three points are collinear. The collinearity follows because both $P_UP_K(z)$ and $C(z)$ lie in the semi-line starting at $z$ and passing through $P_UR_K(z)=\frac{1}{2}(R_K(z)+R_UR_K(z))$. In fact, this semi-line contains the circumcenter $C(z)$ as it is a bisector of the isosceles triangle of vertices $z$, $R_K(z)$ and $R_UR_K(z)$. Bearing in mind that $P_U$ is a linear operator \cite[Proposition 2.10]{Bauschke:2018wa} and that $z\in U$, we get
\[\begin{aligned}z_{\MAP} & = P_UP_K(z) = P_U\left(\frac{1}{2}(z+R_K(z))\right)= \frac{1}{2}P_U(z+R_K(z)) \\
                       &=\frac{1}{2}\left(P_U(z)+P_UR_K(z)\right)=\frac{1}{2}(z+P_UR_K(z)).
\end{aligned}
\] 
Hence, $z_{\MAP}$ is a convex combination of $z$ and $P_UR_K(z)$ with parameter $\frac{1}{2}$. Therefore, we have more than just collinearity of $z, C(z)$ and  $z_{\MAP}$. Both $ C(z)$ and  $z_{\MAP}$ lie on the semi-line starting at $z$ and passing through $P_UR_K(z)$. In particular, this means that $z$ cannot lie between $C(z)$ and $z_{\MAP}$. Thus, in view of \eqref{CRMstep_longer_thanMAP}, the only remaining possibility is that  
$z_{\MAP}$ is a convex combination of $C(z)$ and $z$, that is, 
there exists a parameter  $r\in[0,1]$  such that 
\[
z_{\MAP}=rC(z)+(1-r)z,
\]
and
\[
z_{\MAP}-C(z)=(1-r)(z-C(z)) \mbox{ and } z-z_{\MAP}=r(z-C(z)).
\]
Moreover, the parameter $r$ is strictly larger than zero because otherwise $z$ would be in $K\cap U$. Hence,
\[\label{CRMconvex_combMAP}
z_{\MAP}-C(z)=\frac{1-r}{r}(z-z_{\MAP}).
\]
This equation will be properly combined with the following inner product manipulations
\begin{align}\label{InnerProdManipulation}
(z-z_{\MAP})^T(C(z)-s) &=(z-P_K(z))^T(C(z)-s)+(P_K(z)-z_{\MAP})^T(C(z)-s)  \notag\\
&=\underbrace{(z-P_K(z))^T(C(z)-P_K(z))}_{=0}+\underbrace{(z-P_K(z))^T(P_K(z)-s)}_{\geq 0} \notag\\
& \phantom{= } + \underbrace{(P_K(z)-z_{\MAP})^T(C(z)-s)}_{=0}  \notag\\
& \label{InnerProdManipulation} \geq 0.
\end{align} 
The first under-brace equality follows as a consequence of the right angle at $P_K(z)$ of the triangle of vertices $z$, $P_K(z)$ and $C(z)$. The under-brace inequality is due to characterization of projections onto closed convex sets as, in particular, $s$ belongs to $K$. The last under-brace remark holds since $P_K(z)-z_{\MAP}$ is orthogonal to the subspace $U$ and both $C(z)$ and $s$ are in $U$.  Now, inequality \eqref{InnerProdManipulation} together with equation \eqref{CRMconvex_combMAP} give us
\[
(z_{\MAP}-C(z))^T(s-C(z))=\frac{1-r}{r}(z-z_{\MAP})^T(s-C(z))\leq 0,
\]
which, due to the cosine rule, leads to $\|C(z)-s\|\leq \|z_{\MAP}-s\|$ for any $s\in K\cap U$. Taking this into account and letting $\hat z,\tilde z\in K\cap U$ realize the distance of $C(z), z_{\MAP}$ to $K\cap U$, respectively, we have 
\[\dist(C(z), K\cap U)=\|C(z)-\hat z\|\leq \|C(z)-\tilde z\|\leq  \|z_{\MAP}-\tilde z\|=\dist(z_{\MAP},K\cap U),
\]
proving the first inequalities in items (i) and (ii).

The rest of the proof is a comparison between $z_{\MAP}$ and $z_{\DRM}$. We will see below that under our hypothesis $z_{\MAP}$ is precisely the midpoint of $P_K(z)$ and $z_{\DRM}$. The linearity of $R_U$ will be employed.
\begin{align}R_UP_K(z) &= R_U\left(\frac{1}{2}(z+R_K(z))\right)= \frac{1}{2}(R_U(z)+R_UR_K(z)) \\
&=\frac{1}{2}(z+R_UR_K(z))=z_{\DRM}.
\end{align} 
So, $z_{\MAP}=\frac{1}{2}(P_K(z)+z_{\DRM})$. Furthermore, $\|P_K(z)-z_{\MAP}\|=\|z_{\DRM}-z_{\MAP}\|$ and, $P_K(z)-z_{\MAP}$ and $z_{\DRM}-z_{\MAP}$ are orthogonal to $U$. In particular, for any $s\in K\cap U$ it holds that $(s-z_{\MAP})^T(P_K(z)-z_{\MAP})=0$ and $(s-z_{\MAP})^T(z_{\DRM}-z_{\MAP})=0$. Then, we conclude that the right triangle of vertices $s$, $z_{\MAP}$ and $P_K(z)$ is congruent to the one of vertices $s$, $z_{\MAP}$ and $z_{\DRM}$. This implies that $\|z_{\DRM}-s\|=\|P_K(z)-s\|$. However, $\|P_K(z)-s\|\geq \|z_{\MAP}-s\|$ as $P_K(z)-s$ is a hypotenuse vector with one corresponding leg vector being $z_{\MAP}-s$. Finally, having stated the inequality $\|P_K(z)-s\|\geq \|z_{\MAP}-s\|$ for any $s\in K\cap U$ allows us to take $\check z,\tilde z\in K\cap U$ realizing the distance of $z_{\DRM}, z_{\MAP}$ to $K\cap U$, respectively, and thus we have 
\[\dist(z_{\MAP}, K\cap U)=\|z_{\MAP}-\tilde z\|\leq \|z_{\MAP}-\check z\|\leq  \|z_{\DRM}-\check z\|=\dist(z_{\DRM},K\cap U).\] \qed

\end{proof}

The previous theorem is particularly meaningful when seen as a comparison between CRM and MAP as the associated sequences to these methods stay in the affine subspace $U$ whenever the initial point is taken there. This is not the case for DRM sequences. Anyhow, our numerical section confirms a strict favorability of CRM over both MAP and DRM.

\section{CRM for general convex intersection}
\label{sec:prod}
We are now going to use CRM for solving
\[\label{GeneralConvexProblem}
\text{ Find } x^{*}\in X\coloneqq  \bigcap_{i=1}^m X_i, 
\]
where $X\subset \RR^n$ is a nonempty set given by the intersection of finitely many closed convex sets $X_1, X_2, \ldots, X_m$.

In order to employ CRM let us rewrite Problem \eqref{GeneralConvexProblem} by considering Pierra's reformulation \cite{Pierra:1984hl}. Define $W\coloneqq X_1\times X_2 \times \cdots \times X_m$ as the product space of the sets and $D\coloneqq \{(x,x,\ldots,x)\in  \RR^{nm} \mid  x\in\RR^n\}$. Then, one can easily see that
\[\label{ProductSpacePierra}
 x^{*}\in X \Leftrightarrow ( x^{*}, x^{*},\ldots, x^{*})\in W\cap D.
\]
Due to \eqref{ProductSpacePierra}, solving Problem \eqref{GeneralConvexProblem} corresponds to solving
\[\label{ProductSpaceReformulation}
\text{ Find } z^{*}\in W\cap D. 
\]
Note that it is straightforward to prove that $D$ is a subspace of $\RR^{nm}$. Moreover, considering $m$ arbitrary vectors $x^{(i)}$ in $\RR^n$, with $i=1,\ldots,m$,  we can build an arbitrary point in $\RR^{nm}$ of the form $z=(x^{(1)},x^{(2)}, \, \ldots, x^{(m)})\in\RR^{nm}$ and its projection onto $D$ is given by  
\[\label{eq:projDiagonalSpace}
P_D(z)=\frac{1}{m}\left(\sum_{i=1}^m x^{(i)},\sum_{i=1}^m x^{(i)},\ldots ,\sum_{i=1}^m x^{(i)}\right).\]
As for the orthogonal projection of $z=(x^{(1)},x^{(2)}, \, \ldots, x^{(m)})\in\RR^{nm}$ onto $W$, we have 
\[\label{eq:projProductSpace}
P_W(z)=\left(P_{X_1}(x^{(1)}),P_{X_2}(x^{(2)}),\, \ldots,P_{X_m}(x^{(m)})\right).\]

Next, relying on the results of Section 2, we establish convergence of CRM to a solution of Problem \eqref{GeneralConvexProblem}.

\begin{theorem}[convergence of CRM for general convex intersection]\label{TheoGenConvInt}
Assume $W,D\subset \RR^{n}$ as above and let $x^0\in \RR^n$ be given. Then, taking as initial point $z^0\coloneq (x^0,x^0, \, \ldots, x^0)\in D$, we have that the sequence $\{z^k\}\subset\RR^{nm}$ generated by $z^{k+1}\coloneq \circum\{z^{k},R_{W}(z^{k}),R_{D}R_{W}(z^{k})\}$ is well defined, entirely contained in $D$ and converges to a point $ z^{*}=( x^{*}, x^{*}, \, \ldots,  x^{*})\in\RR^{nm}$, where $ x^{*}\in X$.
\end{theorem} 
\begin{proof} The result follows by enforcing \Cref{TeoremaSpan} with $W$, $D$ and $nm$ playing the role of $K$, $U$ and $n$, respectively.\qed\end{proof}

\section{Numerical experiments}\label{sec:NI}

This section illustrates several numerical comparisons between CRM, MAP and DRM. 
Computational tests were performed on an Intel Xeon W-2133 3.60GHz with 32GB of RAM running Ubuntu 18.04 and using \texttt{Julia} programming language.

In order to present the results of our numerical experiments, we choose the \emph{number of iterations} as performance measure. The rationale of choosing such measure is that in each of the methods, the majority of the computational cost involved in one iteration is equivalent: the same number of orthogonal projections. Moreover, when using the product space reformulation, one can parallelize the computation of projections to speedup the CPU time, and iterations are still equivalent.

\subsection{Intersection between an affine and a convex set}

First, we show the results of the aforementioned methods when applied to solve the problem of finding a point in the intersection of an affine subspace  and a closed convex set. 

In these experiments we want to find $x^{*}\in \RR^{n}$ that solves the conic system  
\[
\label{eq:ConicSystem}
\begin{aligned}
 Ax = b, \\
 x  \in \SOC_{n},
\end{aligned}
\]
where  $A\in\RR^{m\times n}$ is a matrix with  $m<n$, $b\in\RR^{m}$ and $\SOC_{n}$ is the standard \emph{second-order cone of dimension $n$} defined as
\[
\SOC_{n} \coloneqq  \{(t,u)\in \RR^{n}\mid u\in\RR^{n-1}, t\in \RR, \|u\|\leq t \}.
\]
$\SOC_{n}$ is also called the ice-cream cone or the Lorentz cone and it is easy to verify that it is a closed convex set. This problem, called \emph{second-order conic system feasibility},  arises in  second-order cone programming (SOCP)~\cite{Alizadeh:2003,Lobo:1998hq}, where a linear function is minimized over the intersection of an affine set and the intersection of second-order cones and an initial feasible point needs to be found~\cite{Cucker:2015gf}. Of course, $x^{*}$ is a solution of \eqref{eq:ConicSystem} if and only if $x^{*}$ lies in  $\pazocal{S} \coloneqq   \SOC_{n}\cap U_{A,b}$, where $U_{A,b} \coloneqq  \{x\in \RR^{n}\mid Ax =b\}$ is the affine subspace defined by $A$ and $b$, that is, if nonempty, the solution set of  \eqref{eq:ConicSystem}  can be written as  \eqref{eq:KcapUProblem}.

To execute our tests, we randomly generate \num{100} instances of subspaces $U_{A,b}$, where $n$ is fixed as \num{200} and $m$ is a random value between \num{1} and $n-1$. We guarantee that $\pazocal{S}$ is nonempty by  sampling $A$  from the standard normal distribution  and choosing a suitable $b$. Each instance is run for \num{10} initial random points, summing up to \num{1000} individual tests. Each initial point $z^{0}$ is also sampled from the standard normal distribution and is accepted as long as it is not in $\pazocal{S}$. We also assure that the norm of $z^{0}$ is between \num{5} and \num{15} and then we project it over $U_{A,b}$ to begin each method. Thus,  in view of \Cref{TeoremaSpan}, we are sure that CRM solves problem  \eqref{eq:ConicSystem}.

Let $\{z^k\}$ be any  of the three sequences that we monitor: $\{z_\CRM^k\}$, $\{z_\DRM^k\}$ and $\{z_\MAP^k\}$, that is, CRM, DRM and MAP sequences, respectively. We considered as tolerance $\varepsilon \coloneqq  10^{-6}$
and employed as  stopping criteria  the \emph{gap distance}, given by
\[
\|P_{U_{A,b}}(z^k) - P_{\SOC_n}(z^k)\|< \varepsilon,
\]
which we consider a reasonable measure of infeasibility. Note that for CRM and MAP, the sequence monitored lies in $U_{A,b}$, however that is not the case of DRM. So, for the sake of fairness, we utilized the aforementioned stopping criteria for all methods. 
The projections computed to measure the gap distance can be utilized on the next iteration, thus this calculation does not add any extra cost.

The numerical experiments results shown in~\Cref{fig:pp-affineSOC,tab:affineSOC}  corroborate with \Cref{ComparisonCRM_MAP_DRM}, since  CRM has a much better performance than DRM and MAP. \Cref{fig:pp-affineSOC} is a performance profile~\cite{Dolan:2002du}, an analytic tool that allows one to compare several different algorithms on a set of problems with respect to a performance measure or cost.  The vertical axis indicates the percentage of problems solved, while the horizontal axis indicates, in log-scale, the corresponding factor of the number of iterations used by the best solver. It shows that CRM was always better or equal than the other two methods in comparison.

\begin{figure}[!ht]

    \centering
    \includegraphics[width=.7\textwidth]{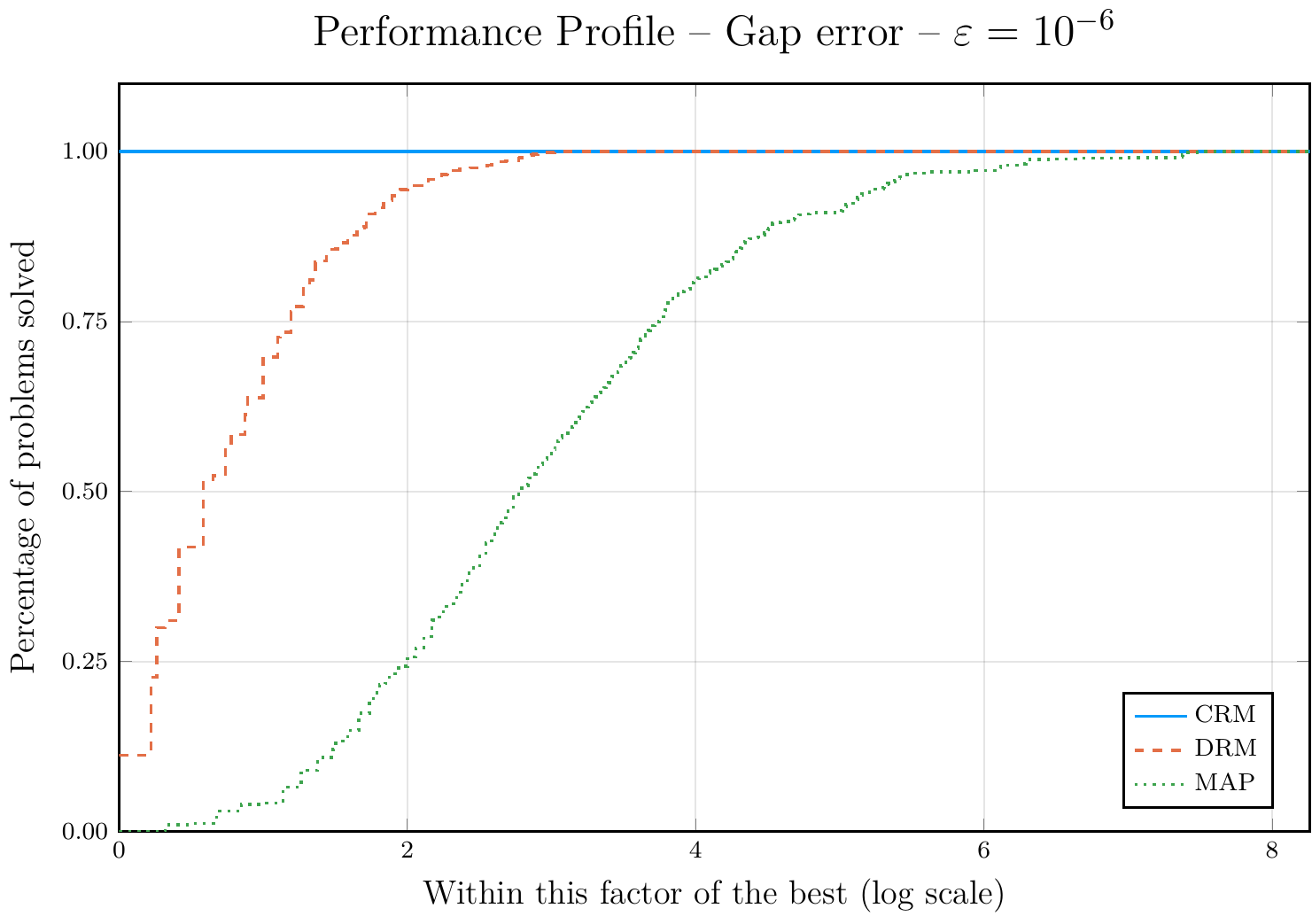}

\caption{Experiments with affine subspaces and the second order cone.\label{fig:pp-affineSOC}}
\end{figure}

\begin{table}[!ht]
\caption{\label{tab:affineSOC} Statistics of the experiments with affine subspaces and the second order cone (in number of iterations).}
\centering 
\begin{tabular}{lcccc}
\toprule
  &    \textbf{mean}  &   \textbf{min} &   \textbf{median} &   \textbf{max} \\
\cmidrule{2-5}
   \textbf{CRM} &    4.727 &   3 &   5.0  &     6    \\
   \textbf{DRM} &   11.602 &   4 &   8.0  &     83  \\
   \textbf{MAP} &   83.981 &   4 &   32.0 &     1063    \\
\bottomrule
\end{tabular}
\end{table}

In \Cref{tab:affineSOC}, each column shows the \emph{mean}, \emph{minimum}, \emph{median} and \emph{maximum}, respectively, of iterations taken by each method for all instances and starting points. We remark that CRM took at most 6 iterations, but in average, less than 5. Moreover, we report that there were 89 (out of 1000) ties between CRM and DRM, while CRM always took less iterations than MAP. There were 10 ties between MAP and DRM.


\subsection{Experiments with product space reformulation}

We show now experiments with Pierra's product space reformulation of CRM stated in \Cref{sec:prod}, which we call CRM-prod. Recall that we defined the Cartesian product of of the convex sets $X_{1},\ldots X_{m}$ as $W$  and defined  $D$ as  the \emph{diagonal} subspace $\{(x,x,\ldots,x)\in  \RR^{nm} \mid  x\in\RR^n\}$.

 Both MAP and DRM can also be considered in product  space reformulation versions. For MAP, the iteration is given by $z^{k+1}_{\MAP}\coloneqq P_{W}P_{D}(z_{\MAP}^{k})$. For DRM, the iteration is given by 
\[
z^{k+1}_{\DRM} \coloneqq \frac{1}{2}z^{k}_{\DRM}  + \frac{1}{2}R_{W}R_{D}(z^{k}_{\DRM})
\]
where we use the projections onto $D$ and $W$ given by \eqref{eq:projDiagonalSpace} and \eqref{eq:projProductSpace} to define the reflectors $R_{D}$ and $R_{W}$, respectively.

We compare the numerical experiments of CRM-prod with MAP-prod and DRM-prod when applied to the problem of \emph{polyhedral feasibility}. Here, each closed convex set is a half-space given by 
\[
X_{i} \coloneqq \{x\in \RR^{n} \mid a_{i}^{T}x \leq b_{i} \}, \text{ for } i=1,\ldots m,
\]
where $a_{i}\in \RR^{n}$ and $b_{i} \in \RR$. $X = \bigcap_{i=1}^{m}X_{i}$ is called\emph{ convex polyhedron} (or polytope).  

To generate the instance, we fix $n=\num{200}$ and sampled each $a_{i}$ from the standard normal distribution, for $i=1,\ldots,m$, where $m$ is randomly selected from $1$ to $n-1$. To assure that $X\neq \varnothing$, we sample from the standard normal distribution an $\bar x$ and fix $\bar b_{i} \coloneqq a_{i}^{T}\bar x$. We then randomly select $p$ indices from $1$ to $m$ to determine the set $\pazocal{I} = \{i_1,\ldots, i_p\}$ and we define
\[
b_{i} = \begin{dcases*}\bar b_{i},&  if $ i\notin\pazocal{I}$\\
                        \bar b_{i} + \|\bar b \|\cdot r,&  if $i\in \pazocal{I}$
\end{dcases*},
\]
where $\bar b = (\bar b_{1},\ldots,\bar b_{m})$ and $r$ is a random value between  \num{0} and \num{1}. Thus, for all $i=1,\ldots,m$, $a_{i}^{T}\bar x \leq b_{i}$ and moreover, for $\hat \imath \in \pazocal{I}$, $a_{\hat \imath}^{T}\bar x < b_{i}$, that is, $X$ has a Slater point.

Again, for each sequence  $\{z^k\}$ that we generate  we use as stopping criteria the error given by the gap distance
\[
\|z^k - P_{W}(z^k)\|< \varepsilon,
\]
with tolerance $\varepsilon \coloneqq  10^{-6}$.  In \Cref{fig:HalfspacesCRM-prod}, we plot number of iterations (horizontal axis) versus the error (vertical axis), both with $\log_{10}$ scale that each method realized to achieve the stopping criteria, for one instance. One can see that CRM-prod overpowers DRM-prod and MAP-prod.

We also restart the methods with 20 different starting points sampled from the standard normal distribution guaranteeing that the norm of each one is  between \num{5} and \num{15} and then we project it over $D$ to begin each method. In \Cref{tab:HalfspacesCRM-prod}, we present  again the \emph{mean}, \emph{minimum}, \emph{median} and \emph{maximum}, respectively, for these experiments.

\begin{figure}[!ht]

    \centering
    \includegraphics[width=.7\textwidth]{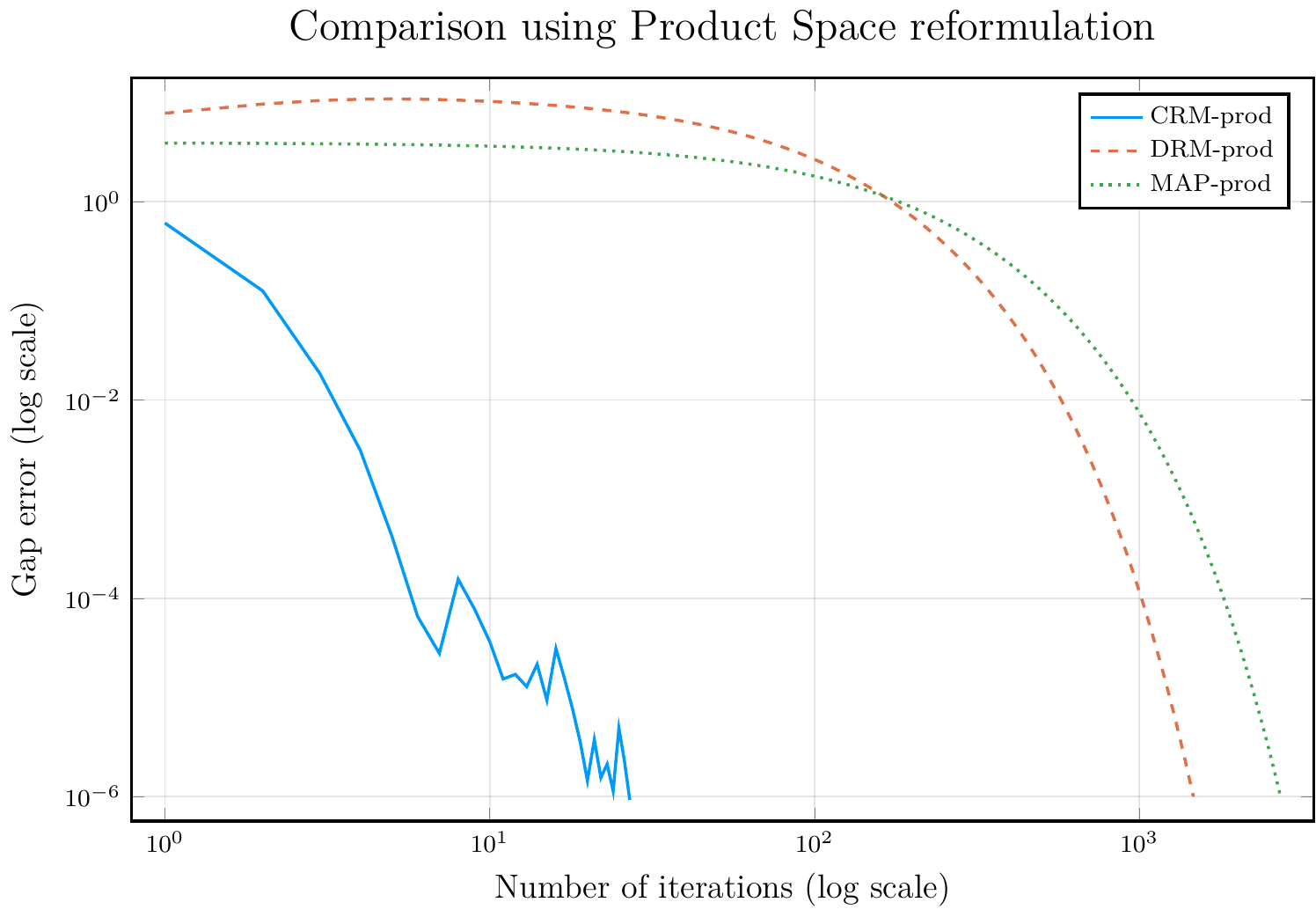}

\caption{Polyhedral feasibility using the product space reformulation.\label{fig:HalfspacesCRM-prod}}
\end{figure}

\begin{table}[!ht]
\caption{\label{tab:HalfspacesCRM-prod} Statistics of the Intersection of half-spaces using the product space reformulation (in number of iterations).}
\centering 
\begin{tabular}{lcccc}
\toprule
  &    \textbf{mean}  &   \textbf{min} &   \textbf{median} &   \textbf{max} \\
\cmidrule{2-5}
   \textbf{CRM} & 41.5     & 19.0    & 38.0    & 89.0     \\
   \textbf{DRM} & 1441.15  & 1036.0  & 1470.5  & 1586.0      \\
   \textbf{MAP} & 2768.3   & 2534.0  & 2787.0  & 2952.0   \\

\bottomrule
\end{tabular}
\end{table}


\section{Concluding remarks}\label{sec:concluding}

This work has taken a large step towards consolidating circumcenter schemes as powerful theoretical and practical tools for addressing feasibility problems. We were able to prove that the circumcentered-reflection method, referred to as CRM throughout the article, finds a point in the intersection of a finite number of closed convex sets. Also, favorable theoretical properties of circumcenters over alternating projections and Douglas-Rachford iterations were proven. Along with the theory, we have seen great numerical performance of CRM in comparison to the classical variants for polyhedral feasibility and inclusion problems involving second order cones. Our aim now is to widen the scope of experiments as well as to broaden the investigation on the use of circumcenters for non-convex problems, including sparse affine feasibility problems and optimization involving manifolds. Finally, due to favorable computational tests, we would like to carry out a local rate convergence analysis for CRM under conditions such as metric-subregularity and H\"older type error bounds.





\bibliographystyle{spmpsci}

\bibliography{refs}

\end{document}